\documentclass[preprint,12pt]{elsarticle}
\usepackage{amsthm,amsfonts,amssymb,amscd,amsmath,enumerate,verbatim,calc,graphicx,geometry}
\usepackage[all]{xy}
\newtheorem{theorem}{Theorem}[section]
\newtheorem{lemma}[theorem]{Lemma}
\newtheorem{proposition}[theorem]{Proposition}
\newtheorem{corollary}[theorem]{Corollary}
\theoremstyle{definition}
\theoremstyle{definitions}

\newtheorem{remark}[theorem]{Remark}
\newtheorem{example}[theorem]{Example}
\theoremstyle{notations}

\theoremstyle{remarks}

\journal{Topology and its Applications}

\begin{document}

\begin{frontmatter}



\title{Topological Coarse Shape Homotopy Groups}


\author[label1]{Fateme~Ghanei}
\ead{fatemeh.ghanei91@gmail.com}
\author[label1]{Hanieh~Mirebrahimi\corref{cor1}}
\ead{h\_mirebrahimi@um.ac.ir}
\author[label1]{Behrooz~Mashayekhy}
\ead{bmashf@um.ac.ir}
\author[label2]{Tayyebe~Nasri}
\ead{tnasri72@yahoo.com}

\address[label1]{Department of Pure Mathematics, Center of Excellence in Analysis on Algebraic Structures, Ferdowsi University of Mashhad,\\
P.O.Box 1159-91775, Mashhad, Iran.}
\address[label2]{Department of Pure Mathematics, Faculty of Basic Sciences,  University of Bojnord,\\
 Bojnord, Iran.}
\cortext[cor1]{Corresponding author}

\begin{abstract}
Cuchillo-Ibanez et al. introduced a  topology on the sets of  shape morphisms between arbitrary topological spaces in 1999.  In this paper, applying a similar idea, we introduce a topology on the set of coarse shape morphisms $Sh^*(X,Y)$, for arbitrary topological spaces $X$ and
$Y$. In particular,  we can consider  a topology on the coarse  shape homotopy group of a topological space $(X,x)$, $Sh^*((S^k,*),(X,x))=\check{\pi}_k^{*}(X,x)$, which makes it a Hausdorff topological group. Moreover, we study some properties of these topological  coarse shape homotopoy groups such as second countability, movability and in particullar, we prove that     $\check{\pi}_k^{*^{top}}$  preserves  finite product of compact Hausdorff spaces. Also, we show that for a pointed topological space $(X,x)$, $\check{\pi}_k^{top}(X,x)$ can be embedded in $\check{\pi}_k^{*^{top}}(X,x)$. 
\end{abstract}


\end{frontmatter}

\section{Introduction and Motivation}
Suppose that  $(X,x)$ is a pointed topological space. We know that $\pi_{k}(X,x)$ has a quotient topology induced by the natural map $q:\Omega^k(X,x)\rightarrow \pi_k(X,x)$, where $\Omega^k(X,x)$ is the $k$th loop space of $(X,x)$ with the compact-open topology. With this topology, $\pi_{k}(X,x)$ is a quasitopological group, denoted by $\pi_k^{qtop}(X,x)$  and for some spaces  it becomes a topological group (see \cite{B0,Br,Bra,G1}).

Calcut and McCarthy \cite{CM1} proved that for a path connected and locally path connected space $X$,  $\pi_1^{qtop}(X)$ is a discrete topological group if and only if $X$ is semilocally 1-connected (see also \cite{Br}). Pakdaman et al. \cite{P1} showed that for a locally $(n-1)$-connected space $X$, $\pi_n^{qtop}(X,x)$ is discrete if and only if $X$ is semilocally n-connected at $x$ (see also \cite{G1}).
Fabel \cite{F2,F3} and Brazas \cite{Br} presented some spaces for which their quasitopological homotopy groups are not topological groups.
 Moreover, despite  Fabel's result \cite{F2} that says the quasitopological fundamental group of the Hawaiian earring is not a topological group, Ghane et al. \cite{G2} proved that the topological $n$th homotopy group of an $n$-Hawaiian like space is a prodiscrete metrizable topological group, for all $n\geq 2$.
 
Cuchillo-Ibanez et al. \cite{CM}  introduced a topology on the set of shape morphisms between arbitrary topological spaces $X$, $Y$, $Sh(X, Y)$.  Moszy\'{n}ska \cite{M} 
showed that for a compact Hausdorff space $(X,x)$, the $k$th shape group $\check{\pi}_k(X,x)$, $k\in \Bbb{N}$, is isomorphic to the set $Sh((S^k, *),(X, x))$  and  Bilan \cite{Bi} mentioned that the result can be  extended  for all topological spaces. The authors \cite{NG}, considering the latter topology on the set of shape morphisms between pointed spaces, obtained a topology on the shape  homotopy groups of arbitrary spaces, denoted by $\check{\pi}_k^{top}(X,x)$ and   showed that with this topology, the $k$th shape group $\check{\pi}_k^{top}(X,x)$  is a Hausdorff topological group, for all $k\in \Bbb{N}$. Moreover, they obtained some topological properties of these groups  under some conditions such as movability,  $\mathbb{N}$-compactness and compactness. In particular, they proved that $\check{\pi}_k^{top}$ commutes with finite product of compact Hausdorff spaces. Also, they presented two spaces $X$ and $Y$ with the same shape homotopy groups such that their topological shape homotopy groups are not isomorphic. 

The aim of this paper is  to introduce a topology on the coarse shape homotopy groups $\check{\pi}_k^{*}(X,x)$ and  to provide  some topological properties of these groups. 
First, similar to the techniques in \cite{CM}, we introduce a 
 topology on the set of coarse shape morphisms $Sh^*(X, Y)$, for  arbitrary topological spaces $X$ and $Y$. Several properties of this topology such continuity of the map $\Omega:Sh^*(X,Y)\times Sh^*(Y,Z)\longrightarrow Sh^*(X,Z)$ given by the composition $\Omega(F^*,G^*)=G^*\circ F^*$ and the equality $Sh^*(X,Y)=\displaystyle{\lim_{\leftarrow} Sh^*(X,Y_{\mu})}$, for an HPol-expansion $\mathbf{q}:Y\rightarrow (Y_{\mu},q_{\mu\mu'},M)$ of $Y$, are proved which are usefull to hereinafter results. Moreover, we show that this topology can also be  induced from an ultrametric similar to the process in  \cite{CM2}.
 
 By the above topology, we can consider a topology on the coarse shape homotopy group $\check{\pi}_k^{*^{top}}(X,x)=Sh^*((S^k, *), (X, x))$ which makes it a  Hausdorff topological group, for all $k\in \Bbb{N}$ and any pointed topological  space $(X,x)$. It is  known that if $X$ and $Y$ are compact Hausdorff spaces, then $X\times Y$ is a  product in the coarse shape category \cite[Theorem 2.2]{Na}. In this case, we show that  the $k$th topological coarse shape group commutes with finite product, for all $k\in \mathbb{N}$. Also, we prove that movability of $\check{\pi}_k^{*^{top}}(X,x)$ can be concluded from the movability of $(X,x)$, for  topological space $(X,x)$ with some conditions. 
 As previously mentioned, $\check{\pi}_k(X,x)$  with the topology defined by Cuchillo-Ibanez et al. \cite{CM} on the set of shape morphisms, is a topological group. We show that this topology also coinsides with the topology induced by  $\check{\pi}_k^{*^{top}}(X,x)$ on the subspace  $\check{\pi}_k(X,x)$. 
\section{Preliminaries}
Recall from \cite{UB} some of the main notions concerning the coarse shape category and  pro$^*$-category. Let $\mathcal{T}$ be a category and let $\mathbf{X}=(X_{\lambda},p_{\lambda\lambda'},\Lambda)$ and $\mathbf{Y}=(Y_{\mu},q_{\mu\mu'},M)$ be two inverse systems in the category $\mathcal{T}$. An {\it $S^*$-morphism} of inverse
systems, $(f,f^n_{\mu}): \mathbf{X} \rightarrow \mathbf{Y} $, consists of an index function $f : M \rightarrow\Lambda$ and of a set of $\mathcal{T}$-morphisms $f^n_{\mu}: X_{f(\mu)} \rightarrow Y_{\mu} $, $n\in \Bbb{N}$,  $\mu\in M$, such that for every related pair $\mu\leq \mu'$ in $M$, there exist a $\lambda\in\Lambda$, $\lambda\geq f(\mu),f(\mu')$, and an $n \in \Bbb{N}$ so that for every $n'\geq n$, $$q_{\mu\mu'}f_{\mu'}^{n'}p_{f(\mu')\lambda}= f_{\mu}^{n'}p_{f(\mu)\lambda}.$$
If $M=\Lambda$ and $f=1_{\Lambda}$, then $(1_{\lambda},f_{\lambda}^n)$ is said to be a {\it level $S^*$-morphism}.
The {\it composition} of S$^*$-morphisms $(f,f^n_{\mu}): \mathbf{X} \rightarrow \mathbf{Y} $ and $(g,g^n_{\nu}): \mathbf{Y} \rightarrow \mathbf{Z}=(Z_{\nu},r_{\nu\nu'},N) $ is an S$^*$-morphism $ (h,h^n_{\nu})=(g,g^n_{\nu})(f,f^n_{\mu}): \mathbf{X} \rightarrow \mathbf{Z} $, where $h=fg$ and $h^n_{\nu}=g^n_{\nu}f^n_{g(\nu)}$, for all $n\in\Bbb{N}$. The {\it identity S$^*$-morphism} on $\mathbf{X}$ is an S$^*$-morphism $(1_{\Lambda}, 1_{X_{\lambda}}^n): \mathbf{X} \rightarrow \mathbf{X} $, where $1_{\Lambda}$ is the identity function and $1_{X_{\lambda}}^n=1_{X_{\lambda}}$ in $\mathcal{T}$, for all $n\in \Bbb{N}$ and every $\lambda\in\Lambda$.

An S$^*$-morphism $(f,f^n_{\mu}): \mathbf{X} \rightarrow \mathbf{Y} $ is said to be {\it equivalent} to an S$^*$-morphism $(f',f'^n_{\mu}): \mathbf{X} \rightarrow \mathbf{Y} $, denoted by $(f,f^n_{\mu})\sim (f',f'^n_{\mu})$, provided every $\mu\in M$ admits a $\lambda\in\Lambda$ and $n \in \Bbb{N}$ such that $\lambda\geq f(\mu),f'(\mu)$ and for every $n'\geq n$,
$$f_{\mu}^{n'}p_{f(\mu)\lambda}= f'^{n'}_{\mu}p_{f'(\mu)\lambda}.$$

The relation $\sim$ is an equivalence relation among S$^*$-morphisms of inverse systems in $\mathcal{T}$. The {\it category} pro$^*$-$\mathcal{T}$ has as objects all inverse systems $\mathbf{X}$ in $\mathcal{T}$ and as morphisms all equivalence classes $\mathbf{f^*}=[(f,f^n_{\mu})]$ of S$^*$-morphisms $(f,f^n_{\mu})$. The composition in pro$^*$-$\mathcal{T}$ is well defined by putting
\[\mathbf{g^*f^*}=\mathbf{h^*}=[(h,h^n_{\nu})],\]
where $(h,h^n_{\nu})=(g,g^n_{\nu})(f,f^n_{\mu})=(fg, g^n_{\nu}f^n_{g(\nu)})$.
For every inverse system $\mathbf{X}$ in $\mathcal{T}$, the identity morphism in pro$^*$-$\mathcal{T}$ is $\mathbf{1_X^*}=[(1_{\Lambda}, 1^n_{X_{\Lambda}})]$.

In particular, if $(X)$ and $(Y)$ are two rudimentary inverse systems in HTop, then every set of mappings $f^n:X\rightarrow Y$, $n\in \Bbb{N}$, induces a map $\mathbf{f^*}:(X)\rightarrow(Y)$ in pro$^*$-HTop.

A functor $\underline{\mathcal{J}}= \underline{\mathcal{J}}_{\mathcal{T}} : pro-\mathcal{T} \rightarrow pro^*-\mathcal{T}$ is defined as follows: For every inverse system $\mathbf{X}$ in $\mathcal{T}$, $\underline{\mathcal{J}} (\mathbf{X}) = \mathbf{X}$ and if $\mathbf{f}\in pro-\mathcal{T}(\mathbf{X}, \mathbf{Y} )$ is represented by $( f , f_{\mu})$, then $\underline{\mathcal{J}} ( \mathbf{f} ) = \mathbf{f^*}=[(f,f^n_{\mu})]\in pro^*-\mathcal{T}(\mathbf{X},\mathbf{Y})$ is represented by the S$^*$-morphism $(f,f^n_{\mu})$, where $f^n_{\mu}= f_{\mu}$ for all $\mu\in M$ and $n\in\Bbb{N}$. Since the functor $\underline{\mathcal{J}}$ is faithful, we may consider the category pro-$\mathcal{T}$ as a subcategory of pro$^*$-$\mathcal{T}$.

Let $\mathcal{P}$ be a subcategory of $\mathcal{T}$. A {\it $\mathcal{P}$- expansion} of an object $X$ in $\mathcal{T}$ is a morphism $\mathbf{p} :X\rightarrow \mathbf{X}$ in pro-$\mathcal{T}$, where $\mathbf{X}$ belongs to pro-$\mathcal{P}$ characterised by the following two properties:\\
(E1) For every object $P$ of $\mathcal{P}$ and every map $h:X\rightarrow P$ in $\mathcal{T}$, there is a $\lambda\in \Lambda$ and a map $f:X_{\lambda}\rightarrow P$ in $\mathcal{P}$ such that $fp_{\lambda}=h$;\\
(E2) If $f_0, f_1:X_{\lambda}\rightarrow P$ in $\mathcal{P}$ satisfy $f_0p_{\lambda}=f_1p_{\lambda}$, then there exists a $\lambda'\geq\lambda$ such that $f_0p_{\lambda\lambda'}=f_1p_{\lambda\lambda'}$.

The subcategory $\mathcal{P}$ is said to be {\it pro-reflective} ({\it dense}) subcategory of $\mathcal{T}$ provided that every object $X$ in $\mathcal{T}$ admits a $\mathcal{P}$-expansion $\mathbf{p} :X\rightarrow \mathbf{X}$.

Let $\mathcal{P}$ be a pro-reflective subcategory of $\mathcal{T}$. Let $\mathbf{p} :X\rightarrow \mathbf{X}$ and $\mathbf{p'} :X\rightarrow \mathbf{X'}$ be two $\mathcal{P}$-expansions of the same object $X$ in $\mathcal{T}$, and let $\mathbf{q} : Y \rightarrow \mathbf{Y}$ and $\mathbf{q'} : Y \rightarrow \mathbf{Y'}$ be two $\mathcal{P}$-expansions of the same object $Y$ in $\mathcal{T}$. Then there exist two natural (unique) isomorphisms $\mathbf{i}:\mathbf{X}\rightarrow \mathbf{X}'$ and $\mathbf{j}:\mathbf{Y}\rightarrow \mathbf{Y}'$ in pro-$\mathcal{P}$ with respect to $\mathbf{p}$, $\mathbf{p'}$ and $\mathbf{q}$, $\mathbf{q'}$, respectively. Consequently $\underline{\mathcal{J}}(\mathbf{i}):\mathbf{X}\rightarrow \mathbf{X}'$ and $\underline{\mathcal{J}}(\mathbf{j}):\mathbf{Y}\rightarrow \mathbf{Y}'$ are isomorphisms in pro$^*$-$\mathcal{P}$. A morphism $\mathbf{f^*}:\mathbf{X}\rightarrow \mathbf{Y}$ is said to be {\it pro$^*$-$\mathcal{P}$ equivalent} to a morphism $\mathbf{f'^*}:\mathbf{X'}\rightarrow \mathbf{Y'}$, denoted by $\mathbf{f^*}\sim\mathbf{f'^*}$, provided that the following diagram in pro$^*$-$\mathcal{P}$ commutes:
\begin{equation}
\label{dia}\begin{CD}
\mathbf{X}@>\underline{\mathcal{J}}(\mathbf{i})>>\mathbf{X'}\\
@VV \mathbf{f^*}V@V \mathbf{f'^*}VV\\
\mathbf{Y}@>\underline{\mathcal{J}}(\mathbf{j})>>\mathbf{Y'}.
\end{CD}\end{equation}

This is an equivalence relation on the appropriate subclass of Mor(pro$^*$-$\mathcal{P}$). Now, the {\it coarse shape category} Sh$^*_{(\mathcal{T},\mathcal{P})}$ for the pair $(\mathcal{T},\mathcal{P})$ is defined as follows: The objects of Sh$^*_{(\mathcal{T},\mathcal{P})}$ are all objects of $\mathcal{T}$. A morphism $F^*:X\rightarrow Y$ is the pro$^*$-$\mathcal{P}$ equivalence class $<\mathbf{f^*}>$ of a mapping $\mathbf{f^*}:\mathbf{X}\rightarrow \mathbf{Y}$ in pro$^*$-$\mathcal{P}$. The {\it composition} of $F^*=<\mathbf{f^*}>:X\rightarrow Y$ and $G^*=<\mathbf{g^*}>:Y\rightarrow Z$ is defined by the representatives, i.e., $G^*F^*=<\mathbf{g^*}\mathbf{f^*}>:X\rightarrow Z$. The {\it identity coarse shape morphism} on an object $X$, $1_X^*:X\rightarrow X$, is the pro$^*$-$\mathcal{P}$ equivalence class $<\mathbf{1_X}^*>$ of the identity morphism $\mathbf{1_X}^*$ in pro$^*$-$\mathcal{P}$.

The faithful functor $\mathcal{J}= \mathcal{J}_{(\mathcal{T},\mathcal{P})} : Sh_{(\mathcal{T},\mathcal{P})}\rightarrow Sh^*_{(\mathcal{T},\mathcal{P})}$ is defined by keeping objects fixed and whose morphisms are induced by the inclusion functor $\underline{\mathcal{J}}= \underline{\mathcal{J}}_{\mathcal{T}} : pro-\mathcal{P} \rightarrow pro^*-\mathcal{P}$.
\begin{remark}\label{Sh}
Let $\mathbf{p}:X\rightarrow \mathbf{X}$ and $\mathbf{q}:Y\rightarrow \mathbf{Y}$ be $\mathcal{P}$-expansions of $X$ and $Y$ respectively. For every morphism $f:X\rightarrow Y$ in $\mathcal{T}$, there is a unique morphism $\mathbf{f}:\mathbf{X}\rightarrow \mathbf{Y}$ in pro-$\mathcal{P}$ such that the following diagram commutes in pro-$\mathcal{P}$:
\begin{equation}
\label{dia}\begin{CD}
\mathbf{X}@<<\mathbf{p}<X\\
@VV \mathbf{f}V@V fVV\\
\mathbf{Y}@<<\mathbf{q}<Y.
\end{CD}\end{equation}

If we take other $\mathcal{P}$-expansions $\mathbf{p'}:X\rightarrow \mathbf{X'}$ and $\mathbf{q'}:Y\rightarrow \mathbf{Y'}$, we obtain another morphism $\mathbf{f'}:\mathbf{X'}\rightarrow \mathbf{Y'}$ in pro-$\mathcal{P}$ such that $\mathbf{f'}\mathbf{p'}=\mathbf{q'}f$ and so we have $\mathbf{f}\sim\mathbf{f'}$ and hence $\underline{\mathcal{J}}(\mathbf{f})\sim \underline{\mathcal{J}}(\mathbf{f'})$ in pro$^*$-$\mathcal{P}$. Therefore, every morphism $f\in \mathcal{T}(X,Y)$ yields an pro$^*$-$\mathcal{P}$ equivalence class $<\underline{\mathcal{J}}(\mathbf{f})>$, i.e., a coarse shape morphism $F^*:X\rightarrow Y$, denoted by $\mathcal{S}^*(f)$. If we put $\mathcal{S}^*(X)=X$ for every object $X$ of $\mathcal{T}$, then we obtain a functor $\mathcal{S}^*:\mathcal{T}\rightarrow Sh^*$, which is  called the {\it coarse shape functor}.
\end{remark}

Since the homotopy category of polyhedra HPol is pro-reflective (dense) in the homotopy category HTop \cite[Theorem 1.4.2]{MS}, the coarse shape category Sh$^*_{(HTop,HPol)}$=Sh$^*$ is well defined.

\section{A topology on the set of coarse shape morphisms}
Similar to the method of \cite{CM}, we can define  a topology on the set of coarse shape morphisms.  
Let $X$ and $Y$ be topological  spaces. Assume $\mathbf{X}=(X_{\lambda},p_{\lambda\lambda'},\Lambda)$ is an inverse system
in pro-HPol and  $\mathbf{p} : X\rightarrow\mathbf{X}$ is an HPol-expansion of $X$. For every $\lambda\in \Lambda$ and
$F^*\in Sh^*(Y,X)$ put $V^{F^*}_{\lambda}=\{G^*\in Sh^*(Y,X)|\ \  p_{\lambda}\circ F^*= p_{\lambda}\circ G^*\}$. First, we prove the following results.
\begin{proposition}\label{To}
The family $\{V^{F^*}_{\lambda}|\ \  F^*\in Sh^*(Y,X) \ \ \text{and}\ \ \ \lambda\in \Lambda\}$ is a basis for a topology $T_\mathbf{p}$ on $Sh^*(Y,X)$. Moreover, if $\mathbf{p'} : X\rightarrow\mathbf{X'}=(X_{\nu},p_{\nu\nu'}, N)$ is another HPol-expansion of $X$, then the identity map $(Sh^*(Y,X), T_\mathbf{p})\longrightarrow (Sh^*(Y,X), T_{\mathbf{p'}})$ is a homeomorphism which  shows that this topology depends only on $X$ and $Y$.
\end{proposition}
\begin{proof}
We know that  $F^*\in V^{F^*}_{\lambda}$ for every $\lambda\in \Lambda$ and every $F^*\in Sh^*(Y,X)$.  Suppose $F^*, G^*\in Sh^*(Y,X)$ and $\lambda_1,\lambda_2\in \Lambda$ and $H^*\in V^{F^*}_{\lambda_1}\cap V^{G^*}_{\lambda_2}$. Since $H^*\in V^{F^*}_{\lambda_1}$, then $p_{\lambda_1}\circ F^*= p_{\lambda_1}\circ H^*$. We show that $V^{F^*}_{\lambda_1}=V^{H^*}_{\lambda_1}$.
Suppose $K^*\in V^{F^*}_{\lambda_1}$, so $p_{\lambda_1}\circ K^*= p_{\lambda_1}\circ F^*=  p_{\lambda_1}\circ H^*$. Therefore $K^*\in V^{H^*}_{\lambda_1}$ and hence $V^{F^*}_{\lambda_1}\subseteq V^{H^*}_{\lambda_1}$. Conversely, if  $K^*\in V^{H^*}_{\lambda_1}$, then we have $p_{\lambda_1}\circ K^*= p_{\lambda_1}\circ H^*=  p_{\lambda_1}\circ F^*$. So $K^*\in V^{F^*}_{\lambda_1}$ and hence $V^{H^*}_{\lambda_1}\subseteq V^{F^*}_{\lambda_1}$. 
Similarly, since $H^*\in V^{G^*}_{\lambda_1}$, we have $V^{G^*}_{\lambda_1}=V^{H^*}_{\lambda_1}$ and so  $H^*\in V^{H^*}_{\lambda_1}\cap V^{H^*}_{\lambda_2}$. We know that  there exists a $\lambda\in \Lambda$ such that $\lambda\geq \lambda_1, \lambda_2$. We show that $H^*\in V^{H^*}_{\lambda}\subseteq V^{H^*}_{\lambda_1}\cap V^{H^*}_{\lambda_2}$ which completes the proof of the first assertion.

Given $K^*\in V^{H^*}_{\lambda}$. We have $p_{\lambda_1 \lambda}p_{\lambda}=p_{\lambda_1}$ and $p_{\lambda_2 \lambda}p_{\lambda}=p_{\lambda_2}$. Since $K^*\in V^{H^*}_{\lambda}$, so $p_{\lambda}\circ K^*= p_{\lambda}\circ H^*$ and therefore $p_{\lambda_1}\circ K^*= p_{\lambda_1}\circ H^*$. Hence $K^*\in V^{H^*}_{\lambda_1}$. Similarly $K^*\in V^{H^*}_{\lambda_{2}}$ and so $K^*\in V^{H^*}_{\lambda_1}\cap V^{H^*}_{\lambda_2}$. 

Now, suppose that $\mathbf{p'} : X\rightarrow\mathbf{X'}$ is another HPol-expansion of $X$. Then there exists a unique isomorphism $\mathbf{i}:\mathbf{X}\longrightarrow \mathbf{X'}$ given by $(i_{\nu}, \phi)$ such that $\mathbf{i}\circ \mathbf{p}=\mathbf{p'}$. Let $V_{\nu}^{F^*}$ be an arbitrary element in the basis of $T_{\mathbf{p^{\prime}}}$, where $\nu\in N$ and $F^*\in Sh^*(Y,X)$. Then $\phi(\nu)\in \Lambda$. We show that $V_{\nu}^{F^*}\in T_\mathbf{p} $ and  this follows that $T_{\mathbf{p'}}\subseteq T_\mathbf{p}$. Given $G^*\in V_{\nu}^{F^*}$, thus $V_{\nu}^{G^*}=V_{\nu}^{F^*}$.   For each $H^*\in V^{G^*}_{\phi(\nu)}$, we have $p_{\phi(\nu)}\circ G^*= p_{\phi(\nu)}\circ H^*$ and so $p'_{\nu}\circ G^*= p'_{\nu}\circ H^*$. Hence $H^*\in V^{G^*}_{\nu}$ and therefore $G^*\in V_{\phi(\nu)}^{G^*}\subseteq V_{\nu}^{G^*}=V_{\nu}^{F^*}$.   Similarly, one can  show that $T_{\mathbf{p}}\subseteq T_{\mathbf{p'}}$.
\end{proof}

\begin{corollary}\label{Dis}
Let $X\in Obj(HPol)$. Then $Sh^*(Y,X)$ is discrete, for every topological space $Y$.
\end{corollary}
\begin{example}
Let $P=\{*\}$ be a singleton and $Q=\{*\}\dot{\cup} \{*\}$ (disjoint union). Then $card(Sh(P,Q)) = 2$ while $card(Sh^*(P,Q)) =2^{\aleph_0}$ (see \cite[Example 7.4]{UB}). It shows that $Sh(P,Q)$  is a countable discrete space while  $Sh^{*}(P,Q)$  is an uncountable discrete space.
\end{example}
\begin{theorem}
The map $\Omega:Sh^*(X,Y)\times Sh^*(Y,Z)\longrightarrow Sh^*(X,Z)$ given by the composition $\Omega(F^*,G^*)=G^*\circ F^*$ is continuous, for arbitrary topological spaces $X,Y$ and $Z$.
\end{theorem}
\begin{proof}
Consider  HPol-expansions $\mathbf{p} : X\rightarrow\mathbf{X}=(X_{\lambda},p_{\lambda\lambda'}, \Lambda)$, $\mathbf{q} : Y\rightarrow\mathbf{Y}=(Y_{\mu},q_{\mu\mu'},M)$ and $\mathbf{r} : Z\rightarrow\mathbf{Z}=(Z_{\nu},r_{\nu\nu'}, N)$ of $X, Y$ and $Z$, respectively. 
Let $F_0^*\in Sh^*(X,Y)$ and $G_0^*\in Sh^*(Y,Z)$ given by $(f^n_{\mu}, f)$ and $(g^n_{\nu}, g)$, respectively. Let $\nu\in N$ and $G_0^*\circ F_0^*\in V_{\nu}^{G_0^*\circ F_0^*}$. We  show that $\Omega( V_{g(\nu)}^{F_0^*}\times V_{\nu}^{G_0^*})\subseteq  V_{\nu}^{G_0^*\circ F_0^*}$. To do this, we must show that for any $F^*\in V_{g(\nu)}^{F_0^*}$ and $G^*\in V_{\nu}^{G_0^*}$, $r_{\nu}\circ G^*\circ F^*= r_{\nu}\circ G_0^*\circ F_0^*$. Since  $F^*\in V_{g(\nu)}^{F_0^*}$, we have $q_{g(\nu)}\circ F_0^*=q_{g(\nu)}\circ F^*$ and  since  $G^*\in V_{\nu}^{G_0^*}$, we have $r_{\nu}\circ G_0^*=r_{\nu}\circ G^*$. Note that $r_{\nu}\circ G_0^*$ is an $S^*$-morphism given by $(g^n_{\nu}, g\alpha_{\nu})$, where $\alpha_{\nu}:\{\nu\}\longrightarrow N$ is the inclusion map. Define $\alpha:Y_{g(\nu)}\longrightarrow Z_{\nu}$ as an $S^*$-morphism given by $(g^n_{\nu}, \beta_{\nu})$, where $\beta_{\nu}:\{\nu\}\longrightarrow \{g(\nu)\}$. We have $r_{\nu}\circ G_0^*=\alpha\circ q_{g(\nu)}$ and so  $r_{\nu}\circ G_0^*\circ F_0^*= \alpha\circ q_{g(\nu)}\circ F_0^*= \alpha\circ q_{g(\nu)}\circ F^*=r_{\nu}\circ G_0^*\circ F^*= r_{\nu}\circ G^*\circ F^*$.
\end{proof}

The following corollary is an immediate consequence of the above theorem.
\begin{corollary}\label{C}
Let $X$ and $Y$ be topological spaces and let $F^*:X\longrightarrow Y$ be an $S^*$-morphism. Let $Z$ be a topological space and consider $F^*_1:Sh^*(Y,Z)\longrightarrow Sh^*(X,Z)$ and $F^*_2:Sh^*(Z,X)\longrightarrow Sh^*(Z,Y)$ to be defined by $F^*_1(H^*)=H^*\circ F^*$ and $F^*_2(G^*)=F^*\circ G^*$.\\
(i) $F_1^*$ and $F_2^*$ are continuous, $(G^*\circ F^*)_2=G^*_2\circ F^*_2$, $(G^*\circ F^*)_1=F^*_1\circ G^*_1$ and $Id^*_1$ and $Id^*_2$ are the corresponding identity maps.\\
(ii) Assume $Sh^*(X)\geq Sh^*(Y)$. Then $Sh^*(Y,Z)$ is homeomorphic  to a retract of $Sh^*(X,Z)$ and $Sh^*(Z,Y)$ is homeomorphic  to a retract of $Sh^*(Z,X)$, for every topological space $Z$. \\
(iii)  Assume $Sh^*(X)= Sh^*(Y)$. Then $Sh^*(Y,Z)$ is homeomorphic  to $Sh^*(X,Z)$ and $Sh^*(Z,Y)$ is homeomorphic  to $Sh^*(Z,X)$, for every topological space $Z$. \\
\end{corollary}
Now, we want to prove the following theorem which is usefull to study the topological properties of the space of coarse shape morphisms.
\begin{theorem}\label{inv}
Let $X$ and $Y$ be topological spaces and let $\mathbf{p}:X\longrightarrow \mathbf{X}=(X_{\lambda}, p_{\lambda\lambda'},\Lambda)$ and  $\mathbf{q}:Y\longrightarrow \mathbf{Y}=(Y_{\mu}, q_{\mu\mu'},M)$ be  HPol-expansions of $X$ and $Y$, respectively. Take $\mathbf{Sh^*}(X,Y)=(Sh^*(X,Y_{\mu}), (q_{\mu\mu'})_*, M)$ and consider the morphism $\mathbf{q}_*:Sh^*(X,Y)\longrightarrow \mathbf{Sh^*}(X,Y)$ induced by $\mathbf{q}$. Then $\mathbf{q}_*$ is an inverse limit of  $\mathbf{Sh^*}(X,Y)$ in Top.
\end{theorem}
\begin{proof}
Let $Z$ be a topological space and let $\mathbf{g}:Z\longrightarrow  (Sh^*(X,Y_{\mu}), (q_{\mu\mu'})_*, M)$ be a morphism in pro-Top. We must show that there is a unique continuous map $\alpha:Z\longrightarrow Sh^*(X,Y)$ such that $\mathbf{q}_*\circ \alpha=\mathbf{g}$ in pro-Top. We know that $g_{\mu}(z)\in Sh^*(X,Y_{\mu})$, for every $z\in Z$ and $\mu\in M$. Suppose $g_{\mu}(z)=<[(g^n_{\mu, z}, \lambda_{\mu, z})]>$ and define $h^z:M\longrightarrow\Lambda$ by $h^z(\mu)= \lambda_{\mu, z}$. We define $\alpha(z)=<[(g^n_{\mu, z}, h^z)]>$. Since $(q_{\mu\mu'})_*\circ g_{\mu'}=g_{\mu}$, so  for every $z\in Z$,  $((q_{\mu\mu'})_*\circ g_{\mu'})(z)=g_{\mu}(z)$. Thus, there is a $\lambda\geq\lambda_{\mu,z},\lambda_{\mu',z}$ and $n\in\Bbb{N}$ such that for every $n'\geq n$, $q_{\mu\mu'}\circ g^{n^{\prime}}_{\mu',z}\circ p_{\lambda_{\mu',z}\lambda}=  g^{n^{\prime}}_{\mu,z}\circ p_{\lambda_{\mu,z}\lambda}$. It follows that $\alpha(z)$ is an $S^*$-morphism. It is clear that $\mathbf{q}_*\circ\alpha=\mathbf{g}$. To complete the proof, we  show that $\alpha$ is continuous. Let $z\in Z$, $\mu\in M$ and $F^*=\alpha(z)\in V^{F^*}_{\mu}$. we have $\alpha^{-1}(V^{F^*}_{\mu})=\{ z'\in Z : \alpha(z')\in V^{F^*}_{\mu}\}=\{z'\in Z : (q_{\mu})_*\circ \alpha(z')=(q_{\mu})_*\circ \alpha(z)\}=\{ z'\in Z : g_{\mu}(z)=g_{\mu}(z')\}=g_{\mu}^{-1}(g_{\mu}(z))$. Since $Sh^*(X, Y_{\mu})$ is discrete, $\{g_{\mu}(z)\}$ is an open subset of $Sh^*(X, Y_{\mu})$ and since $g_{\mu}$ is continuous, we have $g_{\mu}^{-1}(g_{\mu}(z))$ is open subset of $Z$. It follows that $\alpha$ is continuous.
\end{proof}

\begin{corollary}
Let $X$ and $Y$ be two topological spaces. Then $Sh^*(X,Y)$ is a Tychonoff space.
\end{corollary}

Suppose that $(M, \leq)$ is a directed set. From \cite{CM2}, we denote by $L(M)$ the set of all lower classes in $M$ ordered by inclusion, in which $\Delta \subseteq M$ is called a lower class if for every $\delta\in \Delta$ and $\mu\in M$ with $\mu \leq \delta$, then $\mu\in \Delta$. 
Moreover, 
for any two lower classes $\Delta, \Delta^{\prime}\in L(M)$, we say that $\Delta\leq \Delta^{\prime}$ if and only if $\Delta\supset\Delta^{\prime}$. Then $(L(M), \leq)$ is a partially ordered set with the least element $M$ which is denoted by $0$. Furthermore, $L(M)^*=L(M)-0$ is downward directed (see \cite[proposition 2.1]{CM2}).

Let $X$ be a set and $(\Gamma,\leq)$ be a partial ordered set  with a least element $0$. Recall from \cite{UH} that an ultrametric on $X$ is a map $d:X\times X\rightarrow \Gamma$ such that for all $x,y\in X$ and $\gamma\in \Gamma$, the following hold:
\begin{enumerate}
\item[$1)$]
$d(x,y)=0 \Longleftrightarrow x=y$.
\item[$2)$]
$d(x,y)=d(y,x)$.
\item[$3)$]
if $d(x,y)\leq \gamma$ and $d(y,z)\leq \gamma$, then $d(x,z)\leq \gamma$.
\end{enumerate}

Now, using the same idea as in \cite{CM2}, we can prove the following theorem:
\begin{theorem}
Let $X$ and $Y$ be topological spaces. Assume 
$\mathbf{q}:Y\longrightarrow \mathbf{Y}=(Y_{\mu}, q_{\mu\mu'},M)$ is an   HPol-expansion of
 $Y$. For  every $F^*,G^*\in Sh^*(X,Y)$ take
\[d(F^*,G^*)=\{\mu\in M:q_{\mu}\circ F^*=q_{\mu}\circ G^*\}.\]
Then we have an ultrametric $d:Sh^*(X,Y)\times Sh^*(X,Y)\rightarrow (L(M),\leq)$.
\end{theorem}
\begin{proof}
First, we show that $d(F^*,G^*)$ is a lower class. Suppose $\mu\in d(F^*,G^*)$ and $\mu^{\prime}\in M$ such that $\mu^{\prime}\leq \mu$. Then  $q_{\mu^{\prime}}=q_{\mu^{\prime}\mu}q_{\mu}$ and we have 
\[q_{\mu^{\prime}}\circ F^*=q_{\mu^{\prime}\mu}\circ q_{\mu}\circ F^*=q_{\mu^{\prime}\mu}\circ q_{\mu}\circ G^*=q_{\mu^{\prime}}\circ G^*.\]
It follows that $\mu^{\prime}\in d(F^*,G^*)$. Now, let $F^*,G^*\in Sh^*(X,Y)$ 
such that $d(F^*,G^*)=0$. It is  equivalent  to  $q_{\mu}\circ F^*=q_{\mu}\circ G^*$, for every $\mu\in M$ or equivalently $F^*=G^*$. Other conditions  can also be proved easily.  
\end{proof}

Let $(M,\leq)$ be a directed set and $(L(M),\leq)$ be the corresponding ordered set of lower  classes in $M$. For every $\mu\in M$, consider $\{\mu^{\prime}\in M:\mu\geq \mu^{\prime}\}$ as the lower class generated by $\mu$, which is denote by $[\mu]$ and define $\phi:(M,\leq)\rightarrow (L(M),\leq)$ that maps $\mu$ to $[\mu]$. If $\mu\geq \mu^{\prime}$, then $[\mu]\leq [\mu^{\prime}]$ and $(\phi(M),\leq)$ is a partial ordered set and also is downward directed in $L(M)$ (see \cite{CM2}). 

 Now, we have:
 \begin{proposition}
Let $X$ and $Y$ be topological spaces. Suppose $\mathbf{q}:Y\longrightarrow \mathbf{Y}=(Y_{\mu}, q_{\mu\mu'},M)$ is an   HPol-expansion of
 $Y$. For every $\mu \in M$ and $F^*\in Sh^*(X,Y)$ take
  \[B_{[\mu]}(F^*)=\{G^*\in Sh^*(X,Y): d(F^*,G^*)\leq [\mu]\}.\]
 Then the family $\{B_{[\mu]}(F^*): F^*\in Sh^*(X,Y), \mu\in M\}$ is a basis for a topology in $Sh^*(X,Y)$. Moreover, this topology is independent of the fixed HPol-expansion of $Y$ and it coinsides with the topology defined previously . 
 \end{proposition}
 \begin{proof}
 It is obvious that $F^*\in B_{[\mu]}(F^*)$, for all $\mu\in M$. Suppose $F^*,G^*\in Sh^*(X,Y)$ and $\mu_{1},\mu_{2}\in M$ and $H^*\in B_{[\mu_1]}(F^*)\cap B_{[\mu_2]}(G^*)$. Therefore $d(H^*,F^*)\leq [\mu_1]$ and $d(H^*,G^*)\leq [\mu_2]$.  Let $K^*\in B_{[\mu_1]}(H^*)$, then we have $d(K^*, H^*)\leq [\mu_1]$ and $d(H^*,F^*)\leq [\mu_1]$ and so by the definition of an ultrametric, $d(K^*,F^*)\leq [\mu_1]$. It shows that $K^*\in B_{[\mu_1]}(F^*)$ and $B_{[\mu_1]}(H^*)\subseteq B_{[\mu_1]}(F^*)$. Conversely, we can show that $B_{[\mu_1]}(F^*)\subseteq B_{[\mu_1]}(H^*)$ and hence we have $B_{[\mu_1]}(H^*)=B_{[\mu_1]}(F^*)$. Similarly, we can conclude that $B_{[\mu_2]}(H^*)=B_{[\mu_2]}(G^*)$. Hence, to prove the first assertion, it is enough to consider  $\mu\in M$ such that $\mu\geq \mu_1,\mu_2$, then $[\mu]\leq [\mu_1],[\mu_2]$ and this easily  implies that $H^*\in B_{[\mu]}(H^*)\subseteq B_{[\mu_1]}(H^*)\cap B_{[\mu_2]}(H^*)=B_{[\mu_1]}(F^*)\cap B_{[\mu_2]}(G^*)$. 
 
 Now, suppose that $\mathbf{q'} : Y\rightarrow\mathbf{Y'}=(Y_{\nu},q_{\nu\nu'}, N)$ is another HPol-expansion of $Y$. Then there exists a unique isomorphism $\mathbf{j}:\mathbf{Y}\longrightarrow \mathbf{Y'}$ given by $(j_{\nu}, \phi)$ such that $\mathbf{j}\circ \mathbf{q}=\mathbf{q'}$ (we can assume that $\phi$ is an increasing map). Let $\nu\in N$ and $F^*\in Sh^*(X,Y)$, then $\phi(\nu)\in M$. Given $G^*\in B_{[\nu]}(F^*)$, so by the above argument $B_{[\nu]}(F^*)=B_{[\nu]}(G^*)$. 
 For each $H^*\in B_{[\phi(\nu)]}(G^*)$, we have $d(G^*,H^*)\leq [\phi(\nu)]$, i.e., if $\mu \leq \phi(\nu)$, then $q_{\mu}\circ G^*=q_{\mu}\circ H^*$.  Given $\nu^{\prime}\in N$ such that $\nu^{\prime}\leq \nu$, then $\phi(\nu^{\prime})\leq \phi(\nu)$ and so  $q_{\phi(\nu^{\prime})}\circ G^*=q_{\phi(\nu^{\prime})}\circ H^*$.  It implies that $q'_{\nu^{\prime}}\circ G^*= q'_{\nu^{\prime}}\circ H^*$ and thus  $H^*\in B_{[\nu]}(G^*)$. Therefore $G^*\in B_{[\phi(\nu)]}(G^*)\subseteq B_{[\nu]}(G^*)=B_{[\nu]}(F^*)$ and it follows that the topology corresponding to HPol-expansion $\mathbf{q}$ is stronger than the topology corresponding to HPol-expansion $\mathbf{q'}$. Similarly, we can prove that the converse is true.  
 
 Finally, we want to show that the topology induced by the ultrametric $d$ coinsides with the topology $T_{q}$ studied in Proposition \ref{To}. It is easy to see that $ V_{\mu}^{F^*}=B_{[\mu]}(F^*)$, for every  $\mu\in M$ and $F^*\in Sh^*(X,Y)$
  which completes the proof. 
 \end{proof}
\section{The topological coarse shape homotopy groups}\label{S}
Let $X$ be a topological space and $\mathbf{p} : X\rightarrow\mathbf{X}=(X_{\lambda},p_{\lambda\lambda'}, \Lambda)$ be an HPol-expansion of $X$. We know that the $k$th coarse shape group $\check{\pi}_k^{*}(X,x)$, $k \in \mathbb{N}$, is the set  of all coarse shape morphisms $F^*:(S^{k},*)\rightarrow (X,x)$ with the following binary operation
\[F^*+G^*=< \mathbf{f}^*>+<\mathbf{g}^*>=<\mathbf{f}^*+\mathbf{g}^*>=<[(f_{\lambda}^{n})]+[(g_{\lambda}^{n})]>=<[(f_{\lambda}^{n}+g_{\lambda}^{n})]>,\]
where coarse shape morphisms $F^*$ and $G^*$ are represented by morphisms $\mathbf{f}^*=[(f,f_{\lambda}^{n})]$ and $\mathbf{g}^*=[(g,g_{\lambda}^{n})]:(S^k,*)\rightarrow (\mathbf{X},\mathbf{x})$ in pro$^*$-HPol$_{*}$, respectively (see \cite{Bi}).

Now, we show that $\check{\pi}_k^{*}(X,x)=Sh^*((S^k,*),(X,x))$ with the above topology is a topological group which is denoted by $\check{\pi}_k^{*{^{top}}}(X,x)$, for all $k\in \mathbb{N}$.
\begin{theorem}\label{quasi}
Let $(X,x)$ be a pointed topological space. Then $\check{\pi}_k^{*{^{top}}}(X,x)$ is a topological group, for all $k\in \Bbb{N}$.
\end{theorem}
\begin{proof}
First, we show that $\phi:\check{\pi}_k^{*^{top}}(X,x)\rightarrow \check{\pi}_k^{*^{top}}(X,x)$ given by $\phi(F^*)=F^{*^{-1}}$ is continuous, where  $F^*$ and  $F^{*^{-1}}:(S^k,*)\rightarrow(X,x)$ are represented by $\mathbf{f}^{*}=(f,f^n_{\lambda})$ and  $\mathbf{f}^{*^{-1}}=(f,f^{n^{-1}}_{\lambda}):(S^k,*)\rightarrow(\mathbf{X},\mathbf{x})$, respectively and $f^{n^{-1}}_{\lambda}:(S^k,*)\rightarrow(X_{\lambda},x_{\lambda})$ is the inverse loop of $f^n_{\lambda}$. Let $V_{\lambda}^{F^{*^{-1}}}$ be an open neighbourhood of $F^{*^{-1}}$ in $\check{\pi}_k^{*^{top}}(X,x)$. We know that for any $G^*=<[(g,g^n_{\lambda})]>\in V_{\lambda}^{F^*}$, $p_{\lambda}\circ G^*=p_{\lambda}\circ F^*$.  So there is an $n'\in\Bbb{N}$ such that for any $n\geq n'$,  $g^n_{\lambda}\simeq f^n_{\lambda}$ rel $\{*\}$ by \cite[Claim 1 and Claim 2]{UB}. Then for any $n\geq n'$, $g_{\lambda}^{n^{-1}}\simeq f_{\lambda}^{n^{-1}}$ rel $\{*\}$ and so $p_{\lambda}\circ G^{*^{-1}}=p_{\lambda}\circ F^{*^{-1}}$. Thus  $\phi(G^*)\in  V_{\lambda}^{F^{*^{-1}}}$. Therefore, the map $\phi$ is continuous.

Second, we show that the map $m:\check{\pi}_k^{*^{top}}(X,x)\times \check{\pi}_k^{*^{top}}(X,x)\rightarrow \check{\pi}_k^{*^{top}}(X,x)$ given by $m(F^*,G^*)=F^*+G^*$ is continuous, where $F^*+G^*$ is the coarse shape morphism represented by $\mathbf{f^*}+\mathbf{g^*}=(f, f^n_{\lambda}+g^n_{\lambda}):(S^k,*)\rightarrow(\mathbf{X},\mathbf{x})$ and $f^n_{\lambda}+g^n_{\lambda}$ is the concatenation of paths. Let $V_{\lambda}^{F^*+G^*}$ be an open neighbourhood of $F^*+G^*$ in $\check{\pi}_k^{*^{top}}(X,x)$. For any $(K^*,H^*)\in V_{\lambda}^{F^*}\times V_{\lambda}^{G^*}$, we have $p_{\lambda}\circ (K^*+H^*)=(p_{\lambda}\circ K^*)+ (p_{\lambda}\circ H^*)=(p_{\lambda}\circ F^*)+ (p_{\lambda}\circ G^*)=p_{\lambda}\circ ( F^*+G^*)$. Hence $m(K^*,H^*)\in  V_{\lambda}^{F^*+G^*}$ and so $m$ is continuous.
\end{proof}
Using Corollary \ref{C}, we can conclude the following results:
\begin{corollary}
If $F^*:(X,x)\rightarrow (Y,y)$ is a coarse shape morphism, then $F^*_2:\check{\pi}_k^{*^{top}}(X,x)\rightarrow \check{\pi}_k^{*^{top}}(Y,y)$ is continuous.
\end{corollary}
\begin{corollary}\label{A}
If $(X,x)$ and $(Y,y)$ are two pointed topological spaces and $Sh^*(X,x)=Sh^*(Y,y)$, then $\check{\pi}_k^{*^{top}}(X,x)\cong\check{\pi}_k^{*^{top}}(Y,y)$ as topological groups.
\end{corollary}
\begin{corollary}
For any $k\in \Bbb{N}$, $\check{\pi}_k^{*^{top}}(-)$ is a functor from the pointed coarse shape category of spaces to the category of Hausdorff topological groups.
\end{corollary}
\begin{corollary}\label{Bra}
Let $X$ be a topological space and $\mathbf{p} : X\rightarrow\mathbf{X}=(X_{\lambda},p_{\lambda\lambda'}, \Lambda)$ be an HPol-expansion of $X$. 
By Theorem \ref{inv}, we know that  
$\check{\pi}_k^{*^{top}}(X,x)\cong\displaystyle{\lim_{\leftarrow}\check{\pi}_k^{*^{top}}(X_{\lambda},x_{\lambda})}$ as topological groups, for all $k\in \Bbb{N}$. Since every $\check{\pi}_k^{*^{top}}(X_{\lambda},x_{\lambda})$ is discrete and Hausdorff,   $\check{\pi}_k^{*^{top}}(X,x)$ is  prodiscrete and Hausdorff, for every topological space $(X,x)$.
\end{corollary}
\begin{corollary}
Let  $(X,x)=\displaystyle{\lim_{\leftarrow}(X_i,x_i)}$, where $X_i$'s are compact polyhedra. Then for all $k\in \Bbb{N}$,
\[\check{\pi}_k^{*^{top}}(X,x)\cong\displaystyle{\lim_{\leftarrow}\check{\pi}_k^{*^{top}}(X_i,x_i)}.\]
\end{corollary}
\begin{proof}
It can be proved similar to  the Corollary 3.8 in \cite{NG}.
\end{proof}
\begin{corollary}
Let $\mathbf{p}:(X,x)\rightarrow (\mathbf{X},\mathbf{x})=((X_{\lambda},x_{\lambda}),p_{\lambda\lambda'},\Lambda)$ be an HPol$_*$-expansion of a pointed topological space $(X,x)$. Then the following statements hold for all $k\in \Bbb{N}$:\\
(i) If the cardinal number of $\Lambda$ is $\aleph_0$ and $\check{\pi}_k^{*^{top}}(X_{\lambda},x_{\lambda})$ is second countable for every $\lambda\in \Lambda$, then $\check{\pi}_k^{*^{top}}(X,x)$ is second countable.\\
(ii) If $\check{\pi}_k^{*^{top}}(X_{\lambda},x_{\lambda})$ is totally disconnected for every $\lambda\in \Lambda$, then so is $\check{\pi}_k^{*^{top}}(X,x)$.
\end{corollary}
\begin{proof}
The results follow from the fact that the product and the subspace topologies preserve the properties of being  second countable and totally disconnected.
\end{proof}
\begin{remark}
The authors proved a similar result to the  above corollary for shape homotopy groups \cite[Corollary 3.9]{NG}. Note that in that case, we can omit the assumption of second coutability  of $\pi_{k}^{qtop}(X_{\lambda},x_{\lambda})$, for all $\lambda\in \Lambda$. Indeed, 
 If $X$ is a polyhedron, so $X$ is second countable and hence $\Omega^{k}(X,x)$ is second countable, for all $x\in X$ and $k\in \mathbb{N}$ (see \cite{D}). Since $\pi_{k}^{qtop}(X,x)$ is discrete, then the map $q:\Omega^{k}(X,x)\rightarrow \pi_{k}^{qtop}(X,x)$ is a bi-quotient map and therefore $\pi_{k}^{qtop}(X,x)$ is also second countable, for all $k\in \mathbb{N}$ (see \cite{ME}).
 \end{remark}
Let $X$ be a topological space and let $x_{0},x_{1}\in X$. A coarse shape path in $X$ from $x_{0}$ to $x_{1}$ is a bi-pointed coarse shape morphism $\Omega^{*}:(I,0,1)\rightarrow (X,x_{0},x_{1})$. $X$ is said to be coarse shape path connected, if for every pair $x,x^{\prime}\in X$, there is a coarse shape path from $x$ to $x^{\prime}$. If $X$ is a coarse shape path connected space, then $\check{\pi}_k^{*}(X,x)\cong\check{\pi}_k^{*}(X,x^{\prime})$, for any two points $x,x^{\prime}\in X$ and every $k\in \mathbb{N}$ \cite[Corollary 1]{B}. 

Now, we show that these two groups are isomorphic as topological groups, if $X$ is a coarse shape path connected, paracompact and locally compact space.
\begin{theorem}
Let $X$ be a coarse shape path connected, paracompact and locally compact space. Then  $\check{\pi}_k^{*^{top}}(X,x)\cong\check{\pi}_k^{*^{top}}(X,x^{\prime})$, for every pair $x,x^{\prime}\in X$ and all $k\in \mathbb{N}$.
\end{theorem}
\begin{proof}
If  X is a topological space admitting a metrizable polyhedral resolution and  for a pair  $x,x^{\prime}\in X$ there exists a coarse
shape path in X from $x$ to $x^{\prime}$, then $(X,x)$ and $(X,x^{\prime})$ are isomorphic pointed spaces in $Sh^*_{\star}$ (see \cite[Theorem 3]{B}). 
Since coarse shape path connected, paracompact and locally compact spaces satisfy in  the above condition \cite{B.},  $Sh^{*}(X,x)\cong Sh^{*}(X,x^{\prime})$. Hence by Corollary \ref{A} we have $\check{\pi}_k^{*^{top}}(X,x)\cong\check{\pi}_k^{*^{top}}(X,x^{\prime})$, for every pair $x,x^{\prime}\in X$ and all $k\in \mathbb{N}$.
\end{proof}
\section{Main results}
It is well-known that if  the cartesian product of two spaces $X$ and $Y$ admits an HPol-expansion, which is the cartesian product of HPol-expansions of these space, then $X\times Y$ is a product in the shape category (see \cite{M1}). 
 In this case, the authors showed that the $k$th topological  shape group commutes with finite products, for all $k\in \mathbb{N}$ \cite[Theorem 4.1]{NG}. 
 
 Also, if $X$ and $Y$ admit HPol-expansions $\mathbf{p}:X\rightarrow \mathbf{X}$ and $\mathbf{q}:Y\rightarrow \mathbf{Y}$,  respectively, such that $\mathbf{p}\times \mathbf{q}: X\times Y\rightarrow \mathbf{X}\times \mathbf{Y}$ is an HPol-expansion, then $X\times Y$ is a product in the coarse shape category \cite[Theorem 2.2]{Na}. 
 Marde\v{s}i\'{c} \cite{M1} proved that if $\mathbf{p}:X\rightarrow \mathbf{X}$ and $\mathbf{q}:Y\rightarrow \mathbf{Y}$ are HPol-expansions of compact Hausdorff spaces $X$ and $Y$, respectively, 
 then  $\mathbf{p}\times \mathbf{q}: X\times Y\rightarrow \mathbf{X}\times \mathbf{Y}$ is an HPol-expansion and so in this case, $X\times Y$ is a product in the coarse shape category. 
  
  Now, we show that under the above condition, the $k$th topological coarse shape group  commutes with finite products, for all $k\in \mathbb{N}$. 

\begin{theorem}\label{pro}
If $X$ and $Y$ are coarse shape path connected spaces with HPol-expansions $\mathbf{p}: X\rightarrow \mathbf{X}$ and $\mathbf{q}:Y\rightarrow \mathbf{Y}$   such that $\mathbf{p}\times \mathbf{q}:X\times Y\rightarrow \mathbf{X}\times\mathbf{Y}$ is an HPol-expansion, then $\check{\pi}_{k}^{*^{top}}(X\times Y)\cong \check{\pi}_{k}^{*^{top}}(X)\times \check{\pi}_{k}^{*^{top}}(Y)$, for all $k\in \mathbb{N}$.
\end{theorem}
\begin{proof}
Let $\mathcal{S}^*(\pi_{X}):X\times Y\rightarrow X$ and $\mathcal{S}^*(\pi_{Y}):X\times Y\rightarrow Y$ be the induced coarse shape morphisms of canonical projections and assume that $\phi_{X}:\check{\pi}_{k}^{*^{top}}(X\times Y)\rightarrow \check{\pi}_{k}^{*^{top}}(X)$ and $\phi_{Y}:\check{\pi}_{k}^{*^{top}}(X\times Y)\rightarrow \check{\pi}_{k}^{*^{top}}(Y)$ are the induced continuous homomorphisms by $\mathcal{S}^*(\pi_{X})$  and $\mathcal{S}^*(\pi_{Y})$, respectively. Then the induced homomorphism $\phi: \check{\pi}_{k}^{*^{top}}(X\times Y)\rightarrow \check{\pi}_{k}^{*^{top}}(X)\times \check{\pi}_{k}^{*^{top}}(Y)$ is continuous. Since $X\times Y$ is a product in Sh$^*$, we can define a homomorphism $\psi:\check{\pi}_{k}^{*^{top}}(X)\times \check{\pi}_{k}^{*^{top}}(Y)\rightarrow \check{\pi}_{k}^{*^{top}}(X\times Y)$ by $\psi(F^*,G^*)=\lfloor F^*,G^*\rfloor$, where $\lfloor F^*,G^*\rfloor:S^{k}\rightarrow X\times Y$ is a unique  coarse shape morphism with $\mathcal{S}^*(\pi_{X})(\lfloor F^*,G^*\rfloor)=F^*$ and $\mathcal{S}^*(\pi_{Y})(\lfloor F^*,G^*\rfloor)=G^*$. In fact, if $F^*=\langle \mathbf{f}^*=(f,f_{\lambda}^{n})\rangle$ and $G^*=\langle \mathbf{g}^*=(g,g_{\mu}^{n})\rangle$, then $\lfloor F^*,G^*\rfloor=\langle \lfloor \mathbf{f}^*,\mathbf{g}^*\rfloor\rangle$, where $\lfloor \mathbf{f}^*,\mathbf{g}^*\rfloor$ is given by $\lfloor f,g\rfloor_{\lambda\mu}^{n}=f_{\lambda}^{n}\times g_{\mu}^{n}:S^{k}\rightarrow X_{\lambda}\times Y_{\mu}$. By the proof of \cite[Theorem 2.4]{Na}, the homomorphism $\psi$ is well define and moreover, $\phi \circ \psi=id$ and $\psi \circ \phi=id$.

To complete the proof, it is enough to show that $\psi$ is continuous. Let $\lfloor F^*,G^*\rfloor\in V_{\lambda\mu}^{\lfloor F^*,G^*\rfloor}$ be a basis open in the topology on $\check{\pi}_{k}^{*^{top}}(X\times Y)$. Considering open sets $F^*\in V_{\lambda}^{F^*}$ and $G^*\in V_{\mu}^{G^*}$, we show that $\psi(V_{\lambda}^{F^*}\times V_{\mu}^{G^*})\subseteq V_{\lambda\mu}^{\lfloor F^*,G^*\rfloor}$. Let $H^*\in V_{\lambda}^{F^*}$ and $K^*\in V_{\mu}^{G^*}$, then $p_{\lambda}\circ H^*=p_{\lambda}\circ F^*$ and $q_{\mu}\circ K^*=q_{\mu}\circ G^*$. 
By a straight computation, we can conclude that $p_{\lambda}\times q_{\mu}(\lfloor H^*,K^*\rfloor)=p_{\lambda}\times q_{\mu}(\lfloor F^*,G^*\rfloor)$ 
which implies that $\psi(H^*,K^*)=\lfloor H^*,K^*\rfloor\in V_{\lambda\mu}^{\lfloor F^*,G^*\rfloor}$.
\end{proof}

\begin{theorem}\label{re2}
Let $(X,x)$ be a pointed topological space. Then for all $k\in \Bbb{N}$,\\
(i) If $(X,x)\in HPol_*$, then $\check{\pi}_k^{*^{top}}(X,x)$ is discrete.\\
(ii) If $\mathbf{p}:(X,x)\rightarrow (\mathbf{X},\mathbf{x})=((X_{\lambda},x_{\lambda}),p_{\lambda\lambda'},\Lambda)$ is an HPol$_*$-expansion of $(X,x)$ and $\check{\pi}_k^{*^{top}}(X,x)$ is discrete, then $\check{\pi}_k^{*^{top}}(X,x)\leq\check{\pi}_k^{*^{top}}(X_{\lambda},x_{\lambda})$, for some $\lambda\in \Lambda$.
\end{theorem}
\begin{proof}
(i) This follows from Corollary \ref{Dis}.\\
(ii) Since $\check{\pi}_k^{*^{top}}(X,x)$ is a discrete group,  $\{E_x^*\}$ is an open set of identity point of $\check{\pi}_k^{*^{top}}(X,x)$. Thus $\{E_x^*\}=\cup_{\lambda\in \Lambda_0} V_{\lambda}^{F^*}$, where $\Lambda_0\subseteq\Lambda$. Consider the induced homomorphism ${p_{\lambda}}_*:\check{\pi}_k^{*^{top}}(X,x)\rightarrow \check{\pi}_k^{*^{top}}(X_{\lambda},x_{\lambda})$ given by ${p_{\lambda}}_*(F^*)=p_{\lambda}\circ F^*$. Let $G^*\in ker {p_{\lambda}}_*$, i.e., $p_{\lambda}\circ G^*=E_{x_{\lambda}}^*=p_{\lambda}\circ E_x^*$. Thus $G^*\in V^{E_x^*}_{\lambda}\subseteq \cup_{\lambda\in \Lambda_0} V_{\lambda}^{F^*}=\{E_x^*\}$ and so $G^*=E_x^*$. Therefore ${p_{\lambda}}_*$ is injective, for all $\lambda\in \Lambda_0$ and $k\in\Bbb{N}$.
\end{proof}
Recall that an inverse system $\mathbf{X}=(X_{\lambda},p_{\lambda\lambda'},\Lambda)$ of pro-HTop is said to be movable if every $\lambda\in \Lambda$ admits a $\lambda'\geq \lambda$ such that each
$\lambda''\geq \lambda$ admits a morphism $r:X_{\lambda'}\rightarrow X_{\lambda''}$ of HTop with $p_{\lambda\lambda''}\circ r\simeq p_{\lambda\lambda'}$. We say that a topological space $X$ is movable  provided that it admits an HPol-expansion $\mathbf{p}:X\rightarrow \mathbf{X}$ such that $\mathbf{X}$ is a movable inverse system of pro-HPol \cite{MS}. We know that  under some conditions, movability can be transferred from a pointed topological space $(X,x)$ to $\check{\pi}_k^{top}(X,x)$ (see \cite{NG}) and  now we show that it can be transferred to $\check{\pi}_k^{*^{top}}(X,x)$ too. 

\begin{lemma}\label{mov}
If $(\mathbf{X},\mathbf{x})=((X_{\lambda},x_{\lambda}),p_{\lambda\lambda'},\Lambda)$ is a movable (uniformly movable) inverse system, then  $\mathbf{Sh^*}((S^k,*),(X,x))= (Sh^*((S^k,*),(X_{\lambda},x_{\lambda})), (p_{\lambda\lambda'})_*,\Lambda)$  is also a movable (uniformly movable) inverse system, for all $k\in \Bbb{N}$.
\end{lemma}
\begin{proof}
Let $\lambda\in \Lambda$. Since $(\mathbf{X},\mathbf{x})$ is a movable inverse system, there is a $\lambda'\geq\lambda$ such that for every $\lambda''\geq\lambda$ there is a map $r:(X_{\lambda'},x_{\lambda'})\rightarrow (X_{\lambda''}, x_{\lambda''})$ such that $p_{\lambda\lambda''}\circ r \simeq p_{\lambda\lambda'}$ rel $\{x_{\lambda'}\}$. We consider $r_*: Sh^*((S^k,*),(X_{\lambda'},x_{\lambda'}))\rightarrow Sh^*((S^k,*),(X_{\lambda''},x_{\lambda''}))$. Hence $(p_{\lambda\lambda''})_*\circ r_* \simeq (p_{\lambda\lambda'})_*$ and so $\mathbf{Sh^*}((S^k,*),(X,x))$ is movable.
\end{proof}
\begin{remark}\label{movre}
Let $(X,x)$ be a movable space. Then there exists an HPol$_*$-expansion $\mathbf{p}:(X,x)\rightarrow (\mathbf{X},\mathbf{x})$ such that $(\mathbf{X},\mathbf{x})$ is a movable inverse system. Suppose $\mathbf{p_*}:Sh^*((S^k,*),(X,x))\rightarrow \mathbf{Sh^*}((S^k,*),(X,x))$ is an HPol$_*$-expansion, then using Lemma \ref{mov}, we can conclude that $\check{\pi}_k^{*^{top}}(X,x)$ is a movable topological group, for all $k\in \Bbb{N}$. 
By Theorem \ref{inv}, if $\mathbf{p}:(X,x)\rightarrow (\mathbf{X},\mathbf{x})$ is an HPol$_*$-expansion of $X$, then $\mathbf{p_*}:Sh^*((S^k,*),(X,x))\rightarrow \mathbf{Sh^*}((S^k,*),(X,x))$ is an inverse limit of $\mathbf{Sh^*}((S^k,*),(X,x))= (Sh^*((S^k,*),(X_{\lambda},x_{\lambda})), (p_{\lambda\lambda'})_*,\Lambda)$. Now, if $Sh^*((S^k,*),(X_{\lambda},x_{\lambda}))$ is a compact  polyhedron for all $\lambda\in \Lambda$, then by \cite[Remark 1]{FZ} $\mathbf{p_*}$ is an HPol$_*$-expansion of  \linebreak $Sh^*((S^k,*),(X,x))$ and therefore in this case, movability of  $(X,x)$ implies movability of  $\check{\pi}_k^{*^{top}}(X,x)$.
\end{remark}
\begin{remark}\label{em}
Suppose $(X,x)$ is a topological space and 
 $\mathbf{p}:(X,x)\rightarrow (\mathbf{X},\mathbf{x})=((X_{\lambda},x_{\lambda}),p_{\lambda\lambda'},\Lambda)$ is an HPol$_*$-expansion of $(X,x)$.
Consider $J:\check{\pi}_k^{{top}}(X,x)\longrightarrow \check{\pi}_k^{*^{top}}(X,x)$  given by $J(F=<(f,f_{\lambda})>)=F^*$, where $F^*=<(f,f^n_{\lambda}=f_{\lambda})>$. Then $J$ is an embedding. To prove this, we show that for each $\lambda\in \Lambda$ and for all $F\in \check{\pi}_k^{{top}}(X,x)$, $J(V^F_{\lambda})=V_{\lambda}^{J(F)}\cap J(\check{\pi}_k^{{top}}(X,x))$. Suppose $G=<(g, g_{\lambda})>\in V^F_{\lambda}$, so $p_{\lambda}\circ G=p_{\lambda}\circ F$ or equivalently $g_{\lambda}\simeq f_{\lambda}$. We know that $J(G)=<g_{\lambda}^n=g_{\lambda}>$ and  $J(F)=<f_{\lambda}^n=f_{\lambda}>$. So for all $n\in \Bbb{N}$, $g^n_{\lambda}\simeq f^n_{\lambda}$ and it follows that $p_{\lambda}\circ J(G)=p_{\lambda}\circ J(F)$. Hence $J(G)\in V_{\lambda}^{J(F)}\cap J(\check{\pi}_k^{{top}}(X,x))$. Conversely, suppose that $G^*=<g'^n_{\lambda}>\in V_{\lambda}^{J(F)}\cap J(\check{\pi}_k^{{top}}(X,x))$. Since $G^*\in J(\check{\pi}_k^{{top}}(X,x))$, there exists a $G\in \check{\pi}_k^{{top}}(X,x)$ such that $J(G)=G^*$. If $G=<g_{\lambda}>$, then $J(G)=<g^n_{\lambda}=g_{\lambda}>$. Since $J(G)=G^*$, we can conclude that there is an $n'\in \Bbb{N}$ such that for every $n\geq n'$,  $g'^n_{\lambda}\simeq g_{\lambda}$. On the other hand, we have $p_{\lambda}\circ G^*=p_{\lambda}\circ J(F)$, i.e., there is an $n''\in \Bbb{N}$ such that for every $n\geq n''$,  $g'^n_{\lambda}\simeq f_{\lambda}$. It follows that for every $\lambda\in \Lambda$, $g_{\lambda}\simeq f_{\lambda}$ and hence $p_{\lambda}\circ G=p_{\lambda}\circ F$. Therefore, $G\in V_{\lambda}^F$ and $G^*=J(G)\in J( V_{\lambda}^F)$. Hence $J( V_{\lambda}^F)= V_{\lambda}^{J(F)}\cap J(\check{\pi}_k^{{top}}(X,x))$ which completes the proof.
\end{remark}
Let $(X,x)$ be a topological space. We know that the induced homomorphism $\phi:\pi_k^{qtop}(X,x)\rightarrow\check{\pi}_k^{top}(X,x)$ is continuous, for all $k\in \mathbb{N}$.  Consider the composition $J\circ \phi:\pi_{k}^{qtop}(X,x)\rightarrow \check{\pi}_k^{*^{top}}(X,x)$ in which $J$ is the embedding defined in Remark \ref{em}. If $(X,x)$ is shape injective, then the homomorphism $\phi$ is an embedding and hence we have an embedding from $\pi_{k}^{qtop}(X,x)$ to $\check{\pi}_k^{*^{top}}(X,x)$.\\

Let $X\subseteq Y$  and $r:Y\rightarrow X$ be a retraction. Consider the inclusion map $j:X\rightarrow Y$. It is known that $j_{*}:\check{\pi}_k^{top}(X,x)\rightarrow \check{\pi}_k^{top}(Y,x)$ is a topological embedding \cite[Theorem 4.2]{NG} and similar to the proof of it, we can conclude that the induced map $j_{*}:\check{\pi}_k^{*^{top}}(X,x)\rightarrow \check{\pi}_k^{*^{top}}(Y,x)$ is also a topological embedding. \\

In follow, we present examples whose topological coarse shape homotopoy groups are not discrete. 
\begin{example}\label{HE}
Let $(HE,p=(0,0))=\displaystyle{\lim_{\leftarrow}(X_i,p_i)}$ be the Hawaiian Earring where $X_j=\vee_{i=1}^jS^1_i$. The first shape homotopy group $\check{\pi}_1^{top}(HE,p)$ is not discrete (see \cite[Example 4.5]{NG}. So the above Remark  follows that $\check{\pi}_1^{*^{top}}(HE,p)$ is not discrete.
\end{example}
\begin{example}
Let $k\in \Bbb{N}$  and let $\mathbf{X}=(X_n, p_{nn+1}, \Bbb{N})$, where $X_n=\prod_{j=1}^nS^k_j$ is the product of $n$ copies of $k$-sphere $S^k$, for all $n\in \Bbb{N}$ and the bonding morphisms of $\mathbf{X}$ are the projection maps. Put $X=\displaystyle{\lim_{\leftarrow} X_n}$. Refer to \cite{NG}, $\check{\pi}_k^{top}(X)\cong \displaystyle{\lim_{\leftarrow}\pi_k^{qtop}(X_n)}\cong \prod \Bbb{Z}$ is not discrete. 
 Since $\check{\pi}_{k}^{top}(X)$ is a subspace of $\check{\pi}_{k}^{*^{top}}(X)$ and it is not discete, then $\check{\pi}_{k}^{*^{top}}(X)$ is not discrete. 
\end{example}








\section*{References}

\bibliography{mybibfile}

\end{document}